\documentclass[11pt]{article}
\usepackage{mathptmx}      
\usepackage{amsmath}
\usepackage{amsfonts}
\usepackage{amssymb}
\usepackage{url}

\newtheorem{theorem}{Theorem}

\newtheorem{proposition}[theorem]{Proposition}
\newtheorem{lemma}[theorem]{Lemma}
\newtheorem{corollary}[theorem]{Corollary}

\newenvironment{proof}{\noindent\textbf{Proof.}}{\hfill $\square$}

\newcommand{\zin}{\langle z\rangle^{-1}}

\DeclareMathOperator{\alf}{alph}
\DeclareMathOperator{\occ}{occ}

\begin{document}

\title{Remark on the identities of the grammic monoid with three generators}
\author{Mikhail V. Volkov}
\date{Ural Federal University, 620000 Ekaterinburg, Russia\\
m.v.volkov@urfu.ru}

\maketitle

\begin{abstract}
Grammic monoids have recently been introduced by Christian Choffrut in terms of the action of the free monoid over a fixed ordered alphabet $X$ on the set of rows of Young tableaux filled with elements from $X$ via Schensted's insertion. For $X=\{a,b,c\}$ with $a<b<c$, Choffrut has identified the grammic monoid on $X$ with the quotient of the plactic monoid on $X$ over the congruence generated by the pair $(bacb,cbab)$. Since $cbab\ne bcab$ in the latter monoid, the quotient is proper. We show that, nevertheless, the plactic and the grammic monoids with three generators satisfy the same identities.
\end{abstract}

\section{Introduction}

The \emph{plactic monoid with three generators} is defined by the presentation
\begin{equation}\label{eq:knuth3}
\left\langle a,b,c\quad \begin{tabular}{|@{\quad}llll} $a b a=b a a,$ & $b a b=b b a,$ & $a c a=c a a,$  & $c a c=c c a$\\[.1ex] $c b b=b c b,$ &  $c b c=c c b,$ & $b a c=b c a,$ & $a c b=c a b$ \end{tabular}\right\rangle.
\end{equation}
We denote the monoid by $M$. For any elements $u,v\in M$, we let $M/(u=v)$ denote the quotient of $M$ over the congruence on $M$ generated by the pair $(u,v)$. In other words, $M/(u=v)$ is the monoid generated by $a,b,c$ subject to the eight relations in \eqref{eq:knuth3} and the extra relation $u=v$.

As in \cite{KO12}, we let $N_2:=M /(b a c b=c b a b)$. (Here and below we use expressions like $A:=B$ to emphasize that $A$ is defined to be $B$.) By \cite[Theorem~1]{Choffrut}, the monoid $N_2$ is isomorphic to the \emph{grammic monoid with three generators} which was defined in \cite{Choffrut} in terms of the action of the free monoid over the alphabet  $\{a,b,c\}$ with $a<b<c$ on the set of rows of Young tableaux filled with $a,b,c$ via Schensted's insertion. Manipulations with Young tableaux are not used in this note, and therefore, we do not reproduce the combinatorial definition of the grammic monoid.

As the grammic monoid $N_2$ is a quotient of the plactic monoid $M$, every identity holding in $M$ holds in $N_2$, too. We are going to show that the converse is also true: $M$ satisfies all identities of $N_2$. Thus, the plactic and the grammic monoids with three generators satisfy the same identities, that is, they are \emph{equationally equivalent}. In addition, we exhibit a few other monoids equationally equivalent to $M$ and $N_2$.

\section{Preliminaries on identities in monoids}

Here we fix terminology and notation related to identities in monoids.

A \emph{word} is a finite sequence of symbols, called \emph{letters}. If $w=x_1\cdots x_n$ where $x_1,\dots,x_n$ are letters, the set $\{x_1,\dots,x_n\}$ is denoted by $\alf(w)$. For a letter $x$, we denote by $\occ(x,w)$ the number of times that $x$ occurs in $w$, that is,
\[
\occ(x,w):=|\{i\in\{1,\dots,n\}\mid x_i=x\}|.
\]
If $\alf(w)=\{y_1,\dots,y_k\}$, we may make this explicit by writing $w$ as $w(y_{1},\dots,y_{k})$.

If $S$ is a monoid, any map $\varphi\colon\alf(w)\to S$ is called a \emph{substitution}. The \emph{value} $w\varphi$ of $w$ under $\varphi$ is the element of $S$ that results from substituting $x\varphi$ for each letter $x\in\alf(w)$ and computing the product in $S$. If $w=w(y_{1},\dots,y_{k})$ and $y_i\varphi=s_i\in S$ for $i=1,\dots,k$, we have $w\varphi=w(s_{1},\dots,s_{k})$.

An \emph{identity} is a pair of words, traditionally written as a formal equality. We use the sign $\bumpeq$ when writing identities, so we write the pair $(w,w')$ as $w\bumpeq w'$, saving the usual equality sign $=$ for equalities of other kinds. An identity $w\bumpeq w'$ is said to be \emph{trivial} if $w=w'$ and \emph{non-trivial} otherwise.

A monoid $S$ \emph{satisfies} $w\bumpeq w'$ (or $w\bumpeq w'$ \emph{holds} in $S$) if $w\varphi=w'\varphi$ for every substitution $\varphi\colon\alf(ww')\to S$, that is, each substitution of~elements in $S$ for letters occurring in $w$ and $w'$ yields equal values to these words. Obviously, the satisfaction of an identity is inherited by forming direct products and taking submonoids and quotient monoids.

An identity $w\bumpeq w'$ is \emph{balanced} if $\occ(x,w)=\occ(x,w')$ for every $x\in\alf(ww')$. The following easy observation is the master key in this note.

\begin{lemma}
\label{lem:localization}
Suppose that a monoid ${S}$ has a submonoid ${T}$ such that ${S}$ is generated by ${T}\cup\{d\}$ for some element $d$ that commutes with every element of ${T}$. Then all balanced identities satisfied by ${T}$ hold in ${S}$ as well.
\end{lemma}

\begin{proof}
Take any balanced identity
\begin{equation}\label{eq:balance}
w(x_{1},\dots,x_{k})\bumpeq w'(x_{1},\dots,x_{k})
\end{equation}
satisfied by ${T}$. Consider any substitution $\varphi\colon\{x_{1},\dots,x_{k}\}\to{S}$ and let $s_i:=x_i\varphi$, $i=1,\dots,k$. Since ${S}$ is generated by ${T}\cup\{d\}$ and $d$ commutes with every element of ${T}$, there exists a representation $s_i=d^{r_i}p_i$ for some $p_i\in{T}$ and $r_i\in\mathbb{Z}_{\ge0}$. (We assume that $d^{r_i}=1$ whenever $r_i=0$.) Such a representation need not be unique; still, if we fix one such a representation for each $i=1,\dots,k$, we have
\begin{align*}
       w(s_{1},\dots,s_{k})&=d^Nw(p_{1},\dots,p_{k}) &&\text{where  $N=\sum\nolimits_{i=1}^k d^{r_i\occ(x_i,w)}$},\\
       w'(s_{1},\dots,s_{k})&=d^{N'}w'(p_{1},\dots,p_{k})&&\text{where  $N'=\sum\nolimits_{i=1}^k d^{r_i\occ(x_i,w')}$}.
\end{align*}
Since the identity \eqref{eq:balance} holds in ${T}$, we have $w(p_{1},\dots,p_{k})=w'(p_{1},\dots,p_{k})$, and since \eqref{eq:balance} is balanced, $\occ(x_i,w)=\occ(x_i,w')$ for each $i=1,\dots,k$, whence $N=N'$. Therefore, $w(s_{1},\dots,s_{k})=w'(s_{1},\dots,s_{k})$. Since the substitution $\varphi$ is arbitrary, we conclude that \eqref{eq:balance} holds in ${S}$.
\end{proof}

\smallskip

Recall that two monoids are said to be \emph{equationally equivalent} if they satisfy the same identities.
\begin{corollary}
\label{cor:equivalence}
If under the condition of Lemma~\ref{lem:localization}, the submonoid ${T}$ has a submonoid isomorphic to the infinite cyclic monoid, then ${T}$ and ${S}$ are equationally equivalent.
\end{corollary}

\begin{proof}
It is well known that every identity satisfied by the infinite cyclic monoid is balanced. Hence, so is every identity satisfied by $T$ for identities are inherited by submonoids. Now Lemma~\ref{lem:localization} ensures that every identity satisfied by ${T}$ holds in ${S}$. The converse is obviously true.
\end{proof}

\section{Quotients of the plactic monoid with three generators and their localizations}

That the monoid $M$ satisfies a non-trivial identity was deduced by Kubat and Okni{\'n}\-ski \cite{KO15} from the structure theory developed in their earlier paper~\cite{KO12}. The present note uses two facts from this theory, which we recall here.

In \cite{KO12}, Kubat and Okni{\'n}ski defined two quotients of the 3-generated plactic monoid: $N_1:=M /(a c=c a)$ and $N_2:=M /(b a c b=c b a b)$, the latter having already been mentioned. Their role is explained by the next fact that readily follows from \cite[Proposition 2.5]{KO12}:
\begin{proposition}
\label{prop:subdirect}
The $3$-generated plactic monoid $M$ is a subdirect product of its quotients $N_1$ and $N_2$.
\end{proposition}

Consider the element $z:=cba$ of $M$. The two features of this element stated in the next lemma are specializations of general properties of central elements in plactic monoids established in \cite{LS1981}.
\begin{lemma}
\label{lem:z}
\emph{(a)} The element $z$ lies in the center of $M$, that is, $zw=wz$ for every $w\in M$. Moreover, the center of $M$ equals the submonoid $\langle z\rangle$ generated by $z$.

\emph{(b)} If $zw = zv$ for some $w,v\in M$, then $w = v$.
\end{lemma}

Let $M\zin$ stand for the \emph{central localization} of the monoid $M$. Lemma~\ref{lem:z}(a) justifies the notation: as the center of $M$ is just the submonoid $\langle z\rangle$, we only need to adjoin inverses to the elements of this submonoid. Thus, every element $w\in M\zin$ can be written as $w=vz^{m}$ for some $v\in M\setminus Mz$ and $m\in\mathbb{Z}$. Lemma~\ref{lem:z}(b) ensures the uniqueness of such a representation and also implies that $M$ embeds in $M\zin$.

Clearly, the images of $z$ in the quotients $N_1$ and $N_2$ of $M$ lie in the centers of these monoids. Slightly abusing notation, we denote each of these images again by $z$ and consider the localizations $N_1\zin$ and $N_2\zin$. One can show that the analogue of Lemma~\ref{lem:z}(b) holds for the monoid $N_i$, $i=1,2$, and therefore, $N_i$ embeds into $N_i\zin$. The following fact is Proposition 2.6(2) from~\cite{KO12}:
\begin{proposition}
\label{prop:isomorphism}
The monoids $N_1\zin$ and  $N_2\zin$ are isomorphic.
\end{proposition}

Notice that the monoids $N_1$ and $N_2$ themselves are not isomorphic.

\section{Equational equivalence}

\begin{theorem}
\label{thm:equivalence}
The monoids $M$, $N_1$, and $N_2$ are pairwise equationally equivalent.
\end{theorem}

\begin{proof}
For each $i=1,2$, the monoid $N_i$ contains submonoids isomorphic to the infinite cyclic monoid; in fact, for any $u\in N_i\setminus\{1\}$, the submonoid $\langle u\rangle$ is infinite. The localization $N_i\zin$ is generated by $N_i\cup\{z^{-1}\}$ and the element $z^{-1}$ commutes with every element of $N_i$. By Corollary~\ref{cor:equivalence}, the monoids $N_i$ and $N_i\cup\{z^{-1}\}$ are equationally equivalent.

Since $N_1\zin\cong N_2\zin$ by Proposition~\ref{prop:isomorphism}, we conclude that $N_1$ and $N_2$ are equationally equivalent. Therefore, if $N_1$ satisfies $w\bumpeq w'$, then so does $N_2$ whence $w\bumpeq w'$ holds in $M$ as $M$ is a subdirect product of $N_1$ and $N_2$ by Proposition~\ref{prop:subdirect}.

Conversely, if an identity holds in the monoid $M$, it holds in every quotient monoid of $M$, including $N_1$ and $N_2$.
\end{proof}

\smallskip

The quotient monoid $M':=M/(z=1)$ relates to the 3-generated plactic monoid $M$ in the same way as the bicyclic monoid $B:=\langle a,b\mid ba=1\rangle$ relates to the 2-generated plactic monoid $\langle a,b\mid a b a=b a a,\ b a b=b b a\rangle$. The monoid $M'$ appears in \cite{KO12,KO15}; in particular, in \cite{KO15} it is proved that $M'$ and $M$ share identities of a certain form. Namely, a pair of words $s,t$ over $\{x,y\}$ is called \emph{reversive} if one obtains the same two words (as a set) when reading $s$ and $t$ backwards. Lemma~2.2 in \cite{KO15} states that the monoids $M$ and $M'$ satisfy the same identities of the form $s\bumpeq t$, where $s, t$ is a reversive pair of words. We show that all restrictions on $s\bumpeq t$ are not essential.

\begin{proposition}
\label{prop:equivalence}
The monoids $M$ and $M'$ are equationally equivalent.
\end{proposition}

\begin{proof}
Applying to $M$ the argument from the first paragraph of the proof of Theorem~\ref{thm:equivalence}, we get the equational equivalence of $M$ and $M\zin$. The following decomposition occurs in \cite{KO15} as formula (4) at p.103:
\begin{equation}\label{eq:direct}
M\zin\cong M'\times\mathbb{Z}.
\end{equation}
Since a monoid identity holds in the group $\mathbb{Z}$ if and only if it is balanced and all identities of $M'$ are balanced, the isomorphism \eqref{eq:direct} implies that $M\zin$ and $M'$ are equationally equivalent. The result now follows by transitivity.
\end{proof}

\medskip

\noindent\textbf{Remark}. In~\cite{KO15}, the isomorphism \eqref{eq:direct} was justified by a reference to~\cite{KO12} though in~\cite{KO12} it was not stated in such an explicit form. In fact, the result came from~\cite[Lemmas 5 and 6]{CO04} where the authors worked with plactic monoid algebras over a field rather than monoids themselves. For the reader's convenience, we include a short self-contained proof.

\smallskip

\begin{proof}
In the relations in~\eqref{eq:knuth3}, each letter $x\in\{a,b,c\}$ occurs the same number of times in both sides of each relation. Hence $\occ(x,u)=\occ(x,u')$ whenever words $u,u'$ over $\{a,b,c\}$ represent the same element of~$M$. Thus, the notation $\occ_a(w)$ makes sense for any element $w\in M$. Now we define a map $\delta\colon M\zin\to M'\times\mathbb{Z}$ as follows: if $M\zin\ni w=vz^{m}$ where $v\in M\setminus Mz$ and $m\in\mathbb{Z}$, then \[
 w\delta:=(\bar v,m+\occ_a(v)),
 \]
where $\bar v$ denotes the image of $v$ in $M'$. Since $v\mapsto\bar v$ is a bijection between $M\setminus Mz$ and $M'$, the map $\delta$ is a bijection, too. If $w'=v'z^{m}$ where $v'\in M\setminus Mz$ and $m'\in\mathbb{Z}$, we have $ww'=vv'z^{m+m'}$. Represent the product $vv'$ as $vv'=uz^{k}$ where $u\in M\setminus Mz$ and $k\in\mathbb{Z}_{\ge0}$ is the number of `new' $z$'s that arise at the junction of $v$ and $v'$. Then $\overline{vv'}=\bar u$ and $\occ(a,vv')=\occ(a,u)+k$ since $a$ occurs in $z=cba$ once. Therefore
\begin{align*}
(ww')\delta&=\bigl(vv'z^{m+m'}\bigr)\delta=\bigl(uz^{m+m'+k}\bigr)\delta\\
           &=(\bar u,m+m'+k+\occ(a,u))=(\overline{vv'},m+m'+\occ(a,vv'))\\
           &=(\bar v\cdot\overline{v'},m+m'+\occ(a,v)+\occ(a,v'))\\
           &=(\bar v,m+\occ(a,v))\cdot(\overline{v'},m'+\occ(a,v'))=w\delta\cdot w'\delta.
\end{align*}
Thus, $\delta$ is an isomorphism.
\end{proof}

\medskip

In a similar fashion, one can verify that for $i=1,2$, the monoid $N_i$ is equationally equivalent to its quotient $N'_i$ over the congruence on $N_i$ generated by the pair $(z,1)$. Thus, all monoids $M$, $M'$, $N_1$, $N_2$, $N'_1$, $N'_2$ are pairwise equationally equivalent.

\section{Conclusion}

It has been observed in the literature that `interesting' monoids coming from different parts of mathematics and having seemingly different nature tend to cluster with respect to their identities. One such cluster of equationally equivalent monoids includes the bicyclic monoid $B$, the 2-generated plactic monoid, the monoid of all upper triangular $2\times 2$-matrices over the tropical semiring\footnote{Recall that the tropical semiring $\mathbb{T}$ is formed by the real numbers augmented with the symbol $-\infty$ under the operations $a\oplus b:=\max\{a,b\}$ and $a\otimes b:=a+b$, for which $-\infty$ plays the role of zero: $a\oplus-\infty=-\infty\oplus a=a$ and $a\otimes-\infty=-\infty\otimes a=-\infty$. A square matrix over $\mathbb{T}$ is said to be upper triangular if its entries below the main diagonal all $-\infty$.}, see \cite{DJK18}, the so-called Chinese monoids \cite{JO11}, and some others. In this note, we meet another cluster grouping around the 3-generated plactic monoid; as follows from a recent result in \cite{JK21}, this cluster also hosts the monoid of all upper triangular $3\times 3$-matrices over the tropical semiring. For an example from the realm of finite monoids, see the author's recent note~\cite{Vo22}. It is very tempting to understand the intrinsic reasons of this clustering phenomenon.

\end{document}